\documentclass[12pt,twoside]{amsart}

\usepackage{amsmath,amsthm,amscd,amssymb,mathrsfs,graphicx,amsfonts,mathrsfs}
\usepackage{amsfonts}
\usepackage{amssymb,enumerate}
\usepackage{amsthm}
\usepackage[all]{xy}
\usepackage{hyperref}
\allowdisplaybreaks[4] \footskip=12pt
\renewcommand{\uppercasenonmath}[1]{}

\numberwithin{equation}{section} \theoremstyle{plain}
\newtheorem*{thm*}{Main Theorem}
\newtheorem{thm}{Theorem}[section]
\newtheorem{cor}[thm]{Corollary}
\newtheorem*{cor*}{Corollary}
\newtheorem{lem}[thm]{Lemma}
\newtheorem*{lem*}{Lemma}

\newtheorem*{fact*}{Fact}

\newtheorem*{nota*}{Notation}
\newtheorem{prop}[thm]{Proposition}
\newtheorem*{prop*}{Proposition}

\newtheorem*{rem*}{Remark}

\newtheorem*{observation*}{Observation}

\newtheorem*{exa*}{Example}

\newtheorem*{df*}{Definition}

\newtheorem*{con*}{Construction}

\renewcommand{\geq}{\geqslant}
\renewcommand{\leq}{\leqslant}

\begin{document}
\begin{center}
{\large  \bf Cofiniteness of generalized local cohomology modules for ideals of small dimension}

\vspace{0.5cm} Xiaoyan Yang and Jiaojiao Lu\\
Department of Mathematics, Northwest Normal University, Lanzhou 730070,
China
E-mails: yangxy@nwnu.edu.cn; 2523620335@qq.com
\end{center}

\bigskip
\centerline { \bf  Abstract}
\leftskip10truemm \rightskip10truemm \noindent Let $\mathfrak{a}$ be an ideal of a commutative noetherian ring $R$ and $M, N$ two finitely generated $R$-modules. By using a spectral sequence argument, it is shown that if either $\mathrm{dim}_RM\leq2$ and $\mathrm{H}^{i}_\mathfrak{a}(N)$ are $\mathfrak{a}$-cofinite for all $i\geq0$, or $\mathrm{H}^{i}_\mathfrak{a}(N)$ is an $\mathfrak{a}$-cofinite module of dimension $\leq1$ for each $i\geq1$, or $q(\mathfrak{a},R)\leq1$, then the $R$-modules $\mathrm{H}^{t}_\mathfrak{a}(M,N)$ are $\mathfrak{a}$-cofinite for all $t\geq0$.\\
\vbox to 0.3cm{}

\leftskip0truemm \rightskip0truemm
\bigskip
\section* { \bf Introduction}
Throughout this paper, $R$ is a commutative noetherian ring
with identity and $\mathfrak{a}$ a proper ideal of $R$.
 Let $M$ be an $R$-module and $i$ a non-negative integer. Grothendieck \cite{Gr} introduced the \emph{ith local
cohomology module} $\mathrm{H}^i_\mathfrak{a}(M)$ of $M$ with respect to $\mathfrak{a}$ as follows:
\begin{center}$\mathrm{H}^i_\mathfrak{a}(M):=\underrightarrow{\textrm{lim}}\textrm{Ext}^i_R(R/\mathfrak{a}^n,M)$,\end{center}
and proved that $\mathrm{Hom}_R(R/\mathfrak{m},\mathrm{H}^i_\mathfrak{m}(M))$ is finitely generated whenever $(R,\mathfrak{m})$ is local. Later Grothendieck \cite{G} asked whether a
similar statement is valid if $\mathfrak{m}$ is replaced by an arbitrary ideal. Hartshorne \cite{H} gave
a counterexample to Grothendieck's question, and then defined that an $R$-module $M$ is \emph{$\mathfrak{a}$-cofinite} if $\mathrm{Supp}_RM\subseteq\mathrm{V}(\mathfrak{a})$ and $\mathrm{Ext}^i_R(R/\mathfrak{a},M)$ are finitely
generated for all $i\geq0$. He also asked when the local cohomology modules
are $\mathfrak{a}$-cofinite. It is an important problem in commutative algebra as if $\mathrm{H}^i_\mathfrak{a}(M)$ is $\mathfrak{a}$-cofinite then the set $\mathrm{Ass}_R\mathrm{H}^i_\mathfrak{a}(M)$
of associated primes
and the Bass numbers $\mu^j_{R}(\mathfrak{p},\mathrm{H}^i_\mathfrak{a}(M))$
are finite for all $\mathfrak{p}\in\mathrm{Spec}R$ and $i,j\geq0$. In the following years, Hartshorne's Question has been high-profile among researchers, and took the following
culminating form.

\vspace{2mm} \noindent{\bf Theorem I.}\label{Th1.4} {\it{Let $M$ be a finitely generated
$R$-module. If either
 $\mathrm{dim}R\leq2$, or $\mathrm{dim}R/\mathfrak{a}\leq1$, or $\mathrm{cd}(\mathfrak{a},R)\leq1$, or $q(\mathfrak{a},R)\leq1$, then the $R$-modules $\mathrm{H}^{t}_\mathfrak{a}(M)$ are $\mathfrak{a}$-cofinite for all $t\geq0$.}}
\vspace{2mm}

For survey of developments on cofiniteness properties of
local cohomology modules, see \cite{B,BN,BNS1,LM,LM1}.

For two $R$-modules $M,N$ and each $i\geq 0$,  the \emph{$i$th generalized local
cohomology module} of $M$ and $N$ with respect to $\mathfrak{a}$ is defined by
\begin{center}$\mathrm{H}^i_\mathfrak{a}(M,N):=\underrightarrow{\textrm{lim}}\textrm{Ext}^i_R(M/\mathfrak{a}^nM,N)$,\end{center}a generalization of the $i$th local cohomology module,
was first introduced by Herzog in his Habilitation \cite{He} and then continued by Suzuki
\cite{S}, Bijan-Zadeh \cite{BZ}, Yassemi \cite{Y}, and some other authors. They studied some basic
 properties of generalized local cohomology
modules, which generalized some known facts of
ordinary local cohomology modules.

 Cofiniteness of generalized
local cohomology modules have been studied by many authors, for example, Divaani-Aazar and Sazeedeh \cite{DS}, Chu \cite{C}, Divaani-Aazar and Hajikarimi
\cite{DH}, Vahidi and Papari-Zarei \cite{VP} and Borna et al. \cite{BSY}. The main aim of this paper is to demonstrate some new facts and improve some older facts about cofiniteness of
generalized local cohomology. More precisely, we show that

\vspace{2mm} \noindent{\bf Theorem II.}\label{Th1.4} {\it{Let $M,N$ be two
$R$-modules with $M$ finitely generated.

$(1)$ If $\mathrm{dim}_RM\leq2$ and $\mathrm{H}^{i}_\mathfrak{a}(N)$ are $\mathfrak{a}$-cofinite for all $i\geq0$, then the $R$-modules $\mathrm{H}^{t}_\mathfrak{a}(M,N)$ are $\mathfrak{a}$-cofinite minimax for all $t\geq0$ $($see Theorem \ref{lem:1.0}$)$.

$(2)$ If $\mathrm{H}^{i}_\mathfrak{a}(N)$ are $\mathfrak{a}$-cofinite and $\mathrm{dim}_R\mathrm{H}^{i}_\mathfrak{a}(N)\leq1$ for all $i\geq0$, then the $R$-modules $\mathrm{H}^{t}_\mathfrak{a}(M,N)$ are $\mathfrak{a}$-cofinite for all $t\geq0$ $($see Theorem \ref{lem:1.5}$)$.

$(3)$ If $q(\mathfrak{a},R)\leq1$ and $N$ is finitely generated, then the $R$-modules $\mathrm{H}^{t}_\mathfrak{a}(M,N)$ are $\mathfrak{a}$-cofinite minimax for all $t\geq0$ $($see Theorem \ref{lem:1.1'}$)$.}}
\vspace{2mm}

In the proof of this theorem does not need the condition $\mathrm{pd}_RM<\infty$. Several corollaries of this result are given, which extends the classical results about cofiniteness of
local cohomology modules at small dimensions to cofiniteness of generalized local cohomology modules.

Next we recall some notions which we will need later.

We write $\mathrm{Spec}R$ for the set of
prime ideals of $R$. For an ideal $\mathfrak{a}$ in $R$, we set
\begin{center}$\mathrm{V}(\mathfrak{a})=\{\mathfrak{p}\in\textrm{Spec}R\hspace{0.03cm}|\hspace{0.03cm}\mathfrak{a}\subseteq\mathfrak{p}\}$.
\end{center}

 Let $M$ be an $R$-module. The \emph{associated prime} of $M$, denoted by $\mathrm{Ass}_RM$, is
the set of prime ideals $\mathfrak{p}$ of $R$
such that there is a cyclic submodule $N$ of $M$ with $\mathfrak{p}=\mathrm{Ann}_RN$. The set of prime ideals $\mathfrak{p}$ so that there exists a cyclic submodule $N$
of $M$ with $\mathfrak{p}\supseteq\mathrm{Ann}_RN$ is well-known to be the \emph{support} of $M$, denoted by $\mathrm{Supp}_RM$, which is equal to the set
 \begin{center}$\{\mathfrak{p}\in\mathrm{Spec}R\hspace{0.03cm}|\hspace{0.03cm}
M_\mathfrak{p}\neq0\}$.\end{center}
A prime ideal $\mathfrak{p}$ is said to be an \emph{attached prime} of $M$ if $\mathfrak{p}=\mathrm{Ann}_{R}(M/L)$ for some submodule $L$ of $M$. The set of attached primes of $M$ is denoted by $\mathrm{Att}_{R}M$. If $M$ is artinian, then $M$ admits a minimal secondary representation
$M=M_{1}+\cdots+M_{r}$ so that $M_{i}$ is $\mathfrak{p}_{i}$-secondary for $i=1,\cdots,r$. In this case, $\mathrm{Att}_{R}M=\{\mathfrak{p}_{1},\cdots,\mathfrak{p}_{r}\}$.

The \emph{arithmetic
rank of $\mathfrak{a}$}, denoted by
$\mathrm{ara}(\mathfrak{a})$, is the least number of elements of $R$ required to generate an ideal which has
the same radical as $\mathfrak{a}$, i.e.,
\begin{center}$\mathrm{ara}(\mathfrak{a})=\mathrm{min}\{n\geq0\hspace{0.03cm}|\hspace{0.03cm}\exists\ a_1,\cdots,a_n\in R\ \textrm{with}\ \mathrm{Rad}(a_1,\cdots,a_n)=\mathrm{Rad}(\mathfrak{a})\}$.\end{center}For an $R$-module $M$, the \emph{arithmetic rank of $\mathfrak{a}$ with respect
to $M$}, denoted by $\mathrm{ara}_M(\mathfrak{a})$, is defined by the arithmetic rank of the ideal $\mathfrak{a}+
\mathrm{Ann}_RM/\mathrm{Ann}_RM$ in the ring $R/\mathrm{Ann}_RM$.

For an arbitrary $R$-module $M$, set
\begin{center}$\mathrm{cd}(\mathfrak{a},M)=\mathrm{sup}\{n\in\mathbb{Z}\hspace{0.03cm}|\hspace{0.03cm}\mathrm{H}_\mathfrak{a}^n(M)\neq0\}$.\end{center}
The \emph{cohomological dimension of $\mathfrak{a}$} is
\begin{center}$\mathrm{cd}(\mathfrak{a},R)=\mathrm{sup}\{\mathrm{cd}(\mathfrak{a},M)\hspace{0.03cm}|\hspace{0.03cm}M\ \textrm{is\ an}\ R\textrm{-module}\}$.\end{center} Following \cite{B}, we denote \begin{center}$q(\mathfrak{a},M)=\mathrm{min}\{n\geq0\hspace{0.03cm}|\hspace{0.03cm}\mathrm{H}^i_\mathfrak{a}(M)\ \textrm{is\ artinian\ for\ all}\ i> n\}$,\end{center}
It is cleat that $q(\mathfrak{a},M)\leq\mathrm{cd}(\mathfrak{a},M)$ for any $R$-module $M$.

\bigskip
\section{\bf Cofiniteness results}
Let $\mathcal{S}$ be a Serre
subcategory of $R$-modules, i.e., a class of $R$-modules such
that for any short exact sequence
$0\rightarrow L\rightarrow M\rightarrow N\rightarrow 0$,
$M\in\mathcal{S}$ if and only if $L,N\in\mathcal{S}$. We begin with the next lemma, which provides a simple proof of \cite[Theorem 2.6]{VA}.

\begin{lem}\label{lem:2.2}{\it{Let $\mathcal{S}$ be a Serre subcategory of $R$-modules, $M$ and $N$ two
$R$-modules with $M$ finitely generated and $t$ be a non-negative integer such that $\mathrm{Ext}^{t-i}_R(M,\mathrm{H}^{i}_\mathfrak{a}(N))\in\mathcal{S}$ for $0\leq i\leq t$.
Then $\mathrm{H}^{t}_\mathfrak{a}(M,N)\in\mathcal{S}$.}}
\end{lem}
\begin{proof} Consider the Grothendieck spectral sequence \begin{center}
$\xymatrix@C=15pt@R=5pt{
 E_2^{p,q}=\mathrm{Ext}^{p}_R(M,\mathrm{H}^{q}_\mathfrak{a}(N))\ar@{=>}[r]_{\ \ \ \ \ \ \  p}&
 \mathrm{H}^{p+q}_\mathfrak{a}(M,N).}$\end{center}
For $r\geq2$ and $0\leq i\leq t$, consider the differential
\begin{center}
$E_r^{t-i-r,i+r-1}\rightarrow E_r^{t-i,i}\rightarrow E_r^{t-i+r,i-r+1}$.\end{center}
 As $E_r^{t-i-r,i+r-1}=0=E_r^{t-i+r,i-r+1}$ for $r\geq t+2$ and $E_{t+2}^{t-i,i}$ is a subquotient of $E_2^{t-i,i}$, it follows that $E_\infty^{t-i,i}=E_{t+2}^{t-i,i}\in\mathcal{S}$. For $t\geq0$, there
exists a finite filtration
\begin{center}$
 0=\Phi^{t+1}H^{t}\subseteq \Phi^{t}H^{t}\subseteq\cdots\subseteq \Phi^{1}H^{t}\subseteq \Phi^{0}H^{t}=H^{t}:=\mathrm{H}^{t}_\mathfrak{a}(M,N)$,
\end{center}such that $\Phi^{p}H^{t}/\Phi^{p+1}H^{t}\cong E_\infty^{p,t-p}$ for $0\leq p\leq t$.  A successive use of the exact sequence
\begin{center}$
 0\rightarrow \Phi^{p+1}H^t\rightarrow \Phi^{p}H^t\rightarrow \Phi^{p}H^t/\Phi^{p+1}H^t\rightarrow0$
\end{center} implies that $\mathrm{H}^{t}_\mathfrak{a}(M,N)\in\mathcal{S}$, as desired.
\end{proof}

\begin{cor}\label{lem:2.00}{\it{Assume that $M,N$ are finitely generated $R$-modules.
 If $\mathrm{dim}R/\mathfrak{a}=0$, then the $R$-modules $\mathrm{H}^{t}_\mathfrak{a}(M,N)$ are artinian $\mathfrak{a}$-cofinite for all $t\geq0$.}}
\end{cor}

Recall that an
$R$-module $M$ is said to be \emph{minimax} if there exists a finitely generated submodule $N$
of $M$ such that $M/N$ is artinian.

\begin{lem}\label{lem:2.2'}{\rm (\cite[Corollary 4.4]{LM}).} {\it{The class of $\mathfrak{a}$-cofinite minimax $R$-modules is a Serre subcategory of $R$-modules.}}
\end{lem}

We now present the first main theorem of this section, which is a generalization of \cite[Corollary 3.8]{HV} and \cite[Corollary 4.5]{VP}.

 \begin{thm}\label{lem:1.0}{\it{Let $M$ be a finitely generated $R$-module with $\mathrm{dim}_RM\leq2$ and $N$ be an $R$-module such that $\mathrm{H}^{i}_\mathfrak{a}(N)$ are $\mathfrak{a}$-cofinite for all $i\geq0$. Then the $R$-modules $\mathrm{H}^{t}_\mathfrak{a}(M,N)$ are $\mathfrak{a}$-cofinite minimax for all $t\geq0$.}}
\end{thm}
\begin{proof} By Lemma \ref{lem:2.2}, it suffices to show that $\mathrm{Ext}^{p}_R(M,\mathrm{H}^{q}_\mathfrak{a}(N))$ are $\mathfrak{a}$-cofinite minimax for all $p,q\geq0$. If $\mathrm{dim}_RM\leq1$, then, by \cite[Theorem 2.3]{AB}, the claim holds. Let $\mathrm{dim}_RM=2$. The exact sequence $0\rightarrow\Gamma_\mathfrak{a}(M)\rightarrow M\rightarrow M/\Gamma_\mathfrak{a}M\rightarrow0$ induces the next exact sequence \begin{center}$\mathrm{Ext}^{p-1}_R(\Gamma_\mathfrak{a}(M),\mathrm{H}^{q}_\mathfrak{a}(N))\rightarrow
\mathrm{Ext}^{p}_R(M/\Gamma_\mathfrak{a}(M),\mathrm{H}^{q}_\mathfrak{a}(N))\rightarrow
\mathrm{Ext}^{p}_R(M,\mathrm{H}^{q}_\mathfrak{a}(N))\rightarrow\mathrm{Ext}^{p}_R(\Gamma_\mathfrak{a}(M),\mathrm{H}^{q}_\mathfrak{a}(N))$.\end{center}
Note that $\mathrm{Ext}^{p}_R(\Gamma_\mathfrak{a}(M),\mathrm{H}^{q}_\mathfrak{a}(N))$ are $\mathfrak{a}$-cofinite minimax for all $p,q\geq0$, we may assume $\Gamma_\mathfrak{a}(M)=0$. Then $\mathfrak{a}\nsubseteq\bigcup_{\mathfrak{p}\in\mathrm{Ass}_RM}\mathfrak{p}$ by \cite[Lemma 2.1.1]{BS}, and there is $x\in \mathfrak{a}$ and an exact sequence
$0\rightarrow M\stackrel{x}\rightarrow M\rightarrow M/xM\rightarrow 0$, which
induces an exact sequence
\begin{center}$\mathrm{Ext}^p_R(M/xM,\mathrm{H}^{q}_\mathfrak{a}(N))
\rightarrow \mathrm{Ext}^p_R(M,\mathrm{H}^{q}_\mathfrak{a}(N))\stackrel{x}\rightarrow\mathrm{Ext}^{p}_R(M,\mathrm{H}^{q}_\mathfrak{a}(N))$.\end{center}
Therefore, for $p,q\geq0$, we have the following exact sequence
\begin{center}$\mathrm{Ext}^p_R(M/xM,\mathrm{H}^{q}_\mathfrak{a}(N))
\rightarrow (0:_{\mathrm{Ext}^p_R(M,\mathrm{H}^{q}_\mathfrak{a}(N))}x)\rightarrow0$.\end{center}
As $\mathrm{dim}_RM/xM\leq1$, it follows from the proof of \cite[Theorem 2.3]{AB} that $\mathrm{Ext}^{p}_R(M/xM,\mathrm{H}^{q}_\mathfrak{a}(N))$ are $\mathfrak{a}$-cofinite artinian, and so $(0:_{\mathrm{Ext}^p_R(M,\mathrm{H}^{q}_\mathfrak{a}(N))}x)$ are $\mathfrak{a}$-cofinite artinian for all $p,q\geq0$. Thus $(0:_{\mathrm{Ext}^p_R(M,\mathrm{H}^{q}_\mathfrak{a}(N))}\mathfrak{a})$ are of finite length for all $p,q\geq0$. Consequently, by \cite[Proposition 4.1]{LM},
$\mathrm{Ext}^{p}_R(M,\mathrm{H}^{q}_\mathfrak{a}(N))$ are $\mathfrak{a}$-cofinite minimax for all $p,q\geq0$. Now it follows
from Lemmas \ref{lem:2.2} and \ref{lem:2.2'} that $\mathrm{H}^{t}_\mathfrak{a}(M,N)$ are $\mathfrak{a}$-cofinite minimax for all $t\geq0$.
\end{proof}

The next theorem is the second main result of this section, which
is a generalization of \cite[Theorem 2.5]{DH} and \cite[Theorem 2.9]{DS}.

\begin{thm}\label{lem:1.5}{\it{Let $M$ be a finitely generated $R$-module and $N$ be an $R$-module such that $\mathrm{H}^{i}_\mathfrak{a}(N)$ are $\mathfrak{a}$-cofinite and $\mathrm{dim}_R\mathrm{H}^{i}_\mathfrak{a}(N)\leq1$ for all $i\geq0$. Then the $R$-modules $\mathrm{H}^{t}_\mathfrak{a}(M,N)$ are $\mathfrak{a}$-cofinite for all $t\geq0$.}}
\end{thm}
\begin{proof} We apply induction on
$r=\mathrm{ara}_M(\mathfrak{a})$. If $r=0$, then $M=(0:_M\mathfrak{a}^n)$ for
some $n>0$, and so $\mathrm{Ext}^{p}_R(M,\mathrm{H}^{q}_\mathfrak{a}(N))$ are $\mathfrak{a}$-cofinite minimax for all $p,q\geq0$ since $\mathrm{Supp}_RM\subseteq\mathrm{V}(\mathfrak{a})$. Hence Lemma \ref{lem:2.2} implies that $\mathrm{H}^{t}_\mathfrak{a}(M,N)$ are $\mathfrak{a}$-cofinite minimax for all $t\geq0$.
Next assume that $r>0$. Then the short exact sequence
\begin{center}$0\rightarrow\Gamma_\mathfrak{a}(M)\rightarrow M\rightarrow M/\Gamma_\mathfrak{a}(M)\rightarrow 0$\end{center}
induces the following exact sequence
\begin{center}$\mathrm{H}^{t-1}_\mathfrak{a}(\Gamma_\mathfrak{a}(M),N)\rightarrow \mathrm{H}^{t}_\mathfrak{a}(M/\Gamma_\mathfrak{a}(M),N)\rightarrow \mathrm{H}^{t}_\mathfrak{a}(M,N)\rightarrow\mathrm{H}^{t}_\mathfrak{a}(\Gamma_\mathfrak{a}(M),N)$.\end{center}
As
$\mathrm{Ann}_RM\subseteq\mathrm{Ann}_RM/\Gamma_\mathfrak{a}(M)$, we have $\mathrm{ara}_{M/\Gamma_\mathfrak{a}(M)}(\mathfrak{a})\leq\mathrm{ara}_M(\mathfrak{a})$. By Lemma \ref{lem:2.2'}, We may assume that $\Gamma_\mathfrak{a}(M)=0$.
For $t\geq0$, let \begin{center}$T_t=\{\mathfrak{p}\in\mathrm{Ass}_R\mathrm{H}^{t}_\mathfrak{a}(M,N)
\hspace{0.03cm}|\hspace{0.03cm}\mathrm{dim}R/\mathfrak{p}=1\}$.\end{center} Then
$T_t\subseteq\{\mathfrak{p}\in\bigcup_{p=0}^t\mathrm{Ass}_R\mathrm{Ext}^{p}_R(M,\mathrm{H}^{t-p}_\mathfrak{a}(N))
 \hspace{0.03cm}|\hspace{0.03cm}\mathrm{dim}R/\mathfrak{p}=1\}$ is a finite set by \cite[Theorem 2.7]{AB} and \cite[Theorem 2.3]{AM}.
For each $\mathfrak{p}\in T$, the $R_\mathfrak{p}$-module $\mathrm{Hom}_{R_\mathfrak{p}}(R_\mathfrak{p}/\mathfrak{a}R_\mathfrak{p},\mathrm{H}^{q}_\mathfrak{a}(N)_\mathfrak{p})$ is finitely generated with $\mathrm{Supp}_{R_\mathfrak{p}}\mathrm{H}^{q}_\mathfrak{a}(N)_\mathfrak{p}\subseteq\mathrm{V}(\mathfrak{p}R_\mathfrak{p})$, it follows from \cite[Proposition 4.1]{LM}
that $\mathrm{H}^{q}_\mathfrak{a}(N)_\mathfrak{p}$ is artinian $\mathfrak{a}R_\mathfrak{p}$-cofinite, and so $\mathrm{Ext}^p_R(M,\mathrm{H}^{q}_\mathfrak{a}(N))_\mathfrak{p}$ are artinian $\mathfrak{a}R_\mathfrak{p}$-cofinite for all $p\geq0$ and $0\leq q\leq t$.
Hence Lemma
\ref{lem:2.2} yields that the $R_{\mathfrak{p}}$-module
$\mathrm{H}^{i}_{\mathfrak{a}R_{\mathfrak{p}}}(M_{\mathfrak{p}},N_{\mathfrak{p}})\cong
\mathrm{H}^{i}_\mathfrak{a}(M,N)_{\mathfrak{p}}$
is artinian $\mathfrak{a}R_{\mathfrak{p}}$-cofinite for $0\leq i\leq t$.
Let $T_t=\{\mathfrak{p}_1,\cdots\mathfrak{,p}_n\}$. By \cite[Lemma 2.5]{BN}, one has $\mathrm{V}(\mathfrak{a}R_{\mathfrak{p}_j})\cap\mathrm{Att}_{R_{\mathfrak{p}_j}}\mathrm{H}^{i}_\mathfrak{a}(M,N)_{\mathfrak{p}_j}
\subseteq\mathrm{V}(\mathfrak{p}_jR_{\mathfrak{p}_j})$ for $j=1,\cdots,n$ and $i=0,\cdots,t$. Set \begin{center}
$\mathrm{U}=\bigcup_{i=0}^t\bigcup_{j=1}^n\{\mathfrak{q}\in\mathrm{Spec}R\hspace{0.03cm}|\hspace{0.03cm}\mathfrak{q}R_{\mathfrak{p}_j}
\in\mathrm{Att}_{R_{\mathfrak{p}_j}}\mathrm{H}^{i}_\mathfrak{a}(M,N)_{\mathfrak{p}_j}$.
\end{center}Then $\mathrm{U}\cap\mathrm{V}(\mathfrak{a})\subseteq T_t$.
For each $i\geq 0$, we have $\mathrm{Ann}_RM\subseteq\mathrm{Ann}_R\mathrm{H}^{i}_\mathfrak{a}(M,N)$,
and hence $(\mathrm{Ann}_RM)R_{\mathfrak{p}_j}\subseteq\mathfrak{q}R_{\mathfrak{p}_j}$ for every $\mathfrak{q}\in\mathrm{U}$. This implies $\mathrm{Ann}_RM\subseteq\mathfrak{q}$ so that $\mathrm{U}\subseteq\mathrm{Supp}_RM$. Since
$\mathrm{ara}_M(\mathfrak{a})=r$, there exists $a_1,\cdots,a_r\in R$ such that \begin{center}$\mathrm{Rad}(\mathfrak{a}+\mathrm{Ann}_RM)=\mathrm{Rad}((a_1,\cdots,a_r)+\mathrm{Ann}_RM)$.
\end{center}
As $\mathfrak{a}\nsubseteq(\bigcup_{\mathfrak{q}\in\mathrm{U}\backslash\mathrm{V}(\mathfrak{a})}\mathfrak{q})\cup
(\bigcup_{\mathfrak{p}\in\mathrm{Ass}_RM}\mathfrak{p})$, it follows that $(a_1,\cdots,a_r)\nsubseteq(\bigcup_{\mathfrak{q}\in\mathrm{U}\backslash\mathrm{V}(\mathfrak{a})}\mathfrak{q})\cup
(\bigcup_{\mathfrak{p}\in\mathrm{Ass}_RM}\mathfrak{p})$,
so \cite[Ex.16.8]{M}
provides an element $b_1\in(a_2,\cdots,a_r)$ such that $x=b_1+a_1\not\in(\bigcup_{\mathfrak{q}\in\mathrm{U}\backslash\mathrm{V}(\mathfrak{a})}\mathfrak{q})\cup
(\bigcup_{\mathfrak{p}\in\mathrm{Ass}_RM}\mathfrak{p})$. Then
\begin{center}
$(a_1,\cdots,a_r)+\mathrm{Ann}_RM/xM=(a_2,\cdots,a_r)+\mathrm{Ann}_RM/xM$.
\end{center}
 Thus $\mathrm{ara}_{M/xM}(\mathfrak{a})\leq r-1$ and for $i\geq 0$,
   the exact sequence
 $0\rightarrow M\stackrel{x}\rightarrow M\rightarrow M/xM\rightarrow0$ induces the following exact sequence \begin{center}$0\rightarrow\mathrm{H}^{i-1}_\mathfrak{a}(M,N)/x\mathrm{H}^{i-1}_\mathfrak{a}(M,N)
 \rightarrow\mathrm{H}^{i}_\mathfrak{a}(M/xM,N)\rightarrow (0:_{\mathrm{H}^{i}_\mathfrak{a}(M,N)}x)\rightarrow0$.\end{center}
Now the induction yields that $\mathrm{H}^{i}_\mathfrak{a}(M/xM,N)$ is $\mathfrak{a}$-cofinite for $i\geq 0$.
For each $\mathfrak{p}\in T$ and $0\leq i\leq t$, from the choice of $x$, one has
\begin{center}$\mathrm{V}(xR_{\mathfrak{p}})\cap\mathrm{Att}_{R_{\mathfrak{p}}}
\mathrm{H}^{i}_{\mathfrak{a}R_{\mathfrak{p}}}(M_{\mathfrak{p}},N_{\mathfrak{p}})
\subseteq\mathrm{V}(\mathfrak{p}R_{\mathfrak{p}})$,\end{center}it follows from \cite[Lemma 2.4]{BN} that $\mathrm{H}^{i}_{\mathfrak{a}R_\mathfrak{p}}(M_\mathfrak{p},N_\mathfrak{p})/x\mathrm{H}^{i}_{\mathfrak{a}R_\mathfrak{p}}(M_\mathfrak{p},N_\mathfrak{p})$ is of finite length for $i\leq t$.
Note that
$\mathrm{Hom}_R(R/\mathfrak{a},\mathrm{H}^{i}_\mathfrak{a}(M,N)/x\mathrm{H}^{i}_\mathfrak{a}(M,N))$ is finitely generated, so $\mathrm{H}^{i}_\mathfrak{a}(M,N)/x\mathrm{H}^{i}_\mathfrak{a}(M,N)$ is $\mathfrak{a}$-cofinite minimax for $i\leq t$ by \cite[Lemma 2.4]{DH}, and then $(0:_{\mathrm{H}^{t}_\mathfrak{a}(M,N)}x)$ is $\mathfrak{a}$-cofinite.
 Finally, by \cite[Corollary 3.4]{LM}, the $R$-module
$\mathrm{H}^{t}_\mathfrak{a}(M,N)$ is $\mathfrak{a}$-cofinite.
As $t$ is arbitrary, we obtain that $\mathrm{H}^{t}_\mathfrak{a}(M,N)$ are $\mathfrak{a}$-cofinite for all $t\geq0$.
\end{proof}

\begin{cor}\label{lem:3.0}{\it{Let $M,N$ be two finitely generated
$R$-modules such that $\mathrm{dim}_RN/\mathfrak{a}N\leq1$. Then the $R$-modules
$\mathrm{H}^{t}_\mathfrak{a}(M,N)$ is $\mathfrak{a}$-cofinite for all $t\geq 0$.}}
\end{cor}
\begin{proof} By \cite[Corollary 2.7]{BN}, $\mathrm{H}^i_\mathfrak{a}(N)$ are $\mathfrak{a}$-cofinite for all $i\geq 0$. As $\mathrm{Supp}_R\mathrm{H}^i_\mathfrak{a}(N)\subseteq\mathrm{Supp}_RN/\mathfrak{a}N$, we have
$\mathrm{dim}_R\mathrm{H}^i_\mathfrak{a}(N)\leq1$ for all $i\geq 0$. Now the assertion follows from Theorem \ref{lem:1.5}.
\end{proof}

Let $M,N$ be two $R$-modules with $M$ finitely generated. Note that the spectral sequence $E_2^{p,q}=\mathrm{H}^{p}_\mathfrak{a}(\mathrm{Ext}^{q}_R(M,N))\Rightarrow
 \mathrm{H}^{p+q}_\mathfrak{a}(M,N)$ in \cite[Theorem 2.2]{MS} may not be convergent, however, the proof of \cite[Theorem 2.2]{MS} used the convergence of this spectral sequence. The following theorem is a more general version of \cite[Theorem 2.8]{DS}, \cite[Theorem 1.2]{MS}, and eliminates the hypothesis $\mathrm{pd}_RM<\infty$ in \cite[Theorems 2.5 and 2.6]{VBG}.

\begin{thm}\label{lem:1.1'}{\it{Let $M,N$ be two finitely generated
$R$-modules and $q(\mathfrak{a},R)\leq1$. Then the $R$-modules $\mathrm{H}^{t}_\mathfrak{a}(M,N)$ are $\mathfrak{a}$-cofinite minimax for all $t\geq0$.}}
\end{thm}
\begin{proof} Consider the Grothendieck spectral sequence \begin{center}
$\xymatrix@C=15pt@R=5pt{
 E_2^{p,q}=\mathrm{Ext}^{p}_R(M,\mathrm{H}^{q}_\mathfrak{a}(N))\ar@{=>}[r]_{\ \ \ \ \ \ \  p}&
 \mathrm{H}^{p+q}_\mathfrak{a}(M,N).}$\end{center} For
 $t\geq 0$, there exists a finite filtration
\begin{center}$
 0=\Phi^{t+1}H^t\subseteq \Phi^{t}H^t\subseteq\cdots\subseteq \Phi^{1}H^t\subseteq \Phi^{0}H^t=H^t:=\mathrm{H}^{t}_\mathfrak{a}(M,N)$,
\end{center}such that $\Phi^{p}H^t/\Phi^{p+1}H^t\cong E_\infty^{p,t-p}$ for $0\leq p\leq t$. Since $E_\infty^{p,q}$
 is a subquotient of $E_2^{p,q}$ for all
$p,q\geq0$, and $E_2^{p,q}$ are $\mathfrak{a}$-cofinite minimax for
all $p\geq0$ and $q\neq1$ by \cite[Theorem 4.10]{B}, it implies that $E_\infty^{p,q}$ are $\mathfrak{a}$-cofinite minimax for all $p\geq0$ and $q\neq1$ by Lemma \ref{lem:2.2'}. For $r\geq2$, consider the differential
\begin{center}$E_r^{t-1-r,r}\xrightarrow{d_r^{t-1-r,r}}E_r^{t-1,1}
\xrightarrow{d_r^{t-1,1}}E_r^{t-1+r,2-r}$.
\end{center}As $E_2^{t-3,2}$ and $E_2^{t+1,0}$ are $\mathfrak{a}$-cofinite minimax, it follows from Lemma \ref{lem:2.2'} that $\mathrm{im}d_2^{t-3,2}$ and $\mathrm{coker}d_2^{t-3,2}$ are $\mathfrak{a}$-cofinite minimax, and so $E_2^{t-1,1}$ is $\mathfrak{a}$-cofinite minimax. For $r\geq3$, one has $E_r^{t-1+r,2-r}=0$, and there is
an exact sequence \begin{center}$0\rightarrow \mathrm{im}d_r^{t-1-r,r}\rightarrow E_{r}^{t-1,1}\rightarrow E_{r+1}^{t-1,1}\rightarrow0$.
\end{center}
Since $E_{r}^{t-1,1}\cong E_{\infty}^{t-1,1}$ for
$r\gg0$ and $\mathrm{im}d_r^{t-1-r,r}$ is $\mathfrak{a}$-cofinite minimax for $r\geq2$, by using the above sequence inductively, one has $E_{\infty}^{t-1,1}$ is $\mathfrak{a}$-cofinite minimax. Now, a successive use of the short exact sequence
\begin{center}$
 0\rightarrow \Phi^{p+1}H^t\rightarrow \Phi^{p}H^t\rightarrow \Phi^{p}H^t/\Phi^{p+1}H^t\rightarrow0$
\end{center} implies that $\mathrm{H}^{t}_\mathfrak{a}(M,N)$ is $\mathfrak{a}$-cofinite minimax.
\end{proof}

\begin{cor}\label{lem:1.1}{\it{Let $M,N$ be two finitely generated
$R$-modules and $\mathrm{cd}(\mathfrak{a},R)\leq1$.
Then the $R$-modules $\mathrm{H}^{t}_\mathfrak{a}(M,N)$ are $\mathfrak{a}$-cofinite minimax for all $t\geq0$.}}
\end{cor}
\begin{proof} As $\mathrm{cd}(\mathfrak{a},R)\leq1$, one has $q(\mathfrak{a},R)\leq1$. Hence the assertion holds by Theorem \ref{lem:1.1'}.
\end{proof}

The following corollary is a generalization of \cite[Corollaries 2.5 and 2.7]{HV}.

\begin{cor}\label{lem:1.2}{\it{Let $M,N$ be two finitely generated
$R$-modules and $\mathfrak{a}$ an ideal of $R$ such that either $\mathrm{dim}R\leq2$, or
 $\mathrm{dim}R/\mathfrak{a}\leq1$, or
 $\mathrm{cd}(\mathfrak{a},R)\leq1$.
Then the set $\mathrm{Ass}_R\mathrm{H}^{t}_\mathfrak{a}(M,N)$ and the Bass number
$\mu^i(\mathfrak{p},\mathrm{H}^{t}_\mathfrak{a}(M,N))$ are finite for all $i,t\geq0$.}}
\end{cor}

\begin{cor}\label{lem:1.3}{\it{Let $M,N$ be two finitely generated
$R$-modules and $\mathfrak{a}$ be an ideal of a local ring $R$ such that either $\mathrm{dim}R=3$ or
 $\mathrm{dim}R/\mathfrak{a}=2$. If $\mathrm{dim}_RM\leq2$,
then
$\mathrm{H}^{t}_\mathfrak{a}(M,N)$ are $\mathfrak{a}$-cofinite if and only if $\mathrm{Hom}_R(R/\mathfrak{a},\mathrm{H}^{t}_\mathfrak{a}(N))$ are finitely generated for all $t\geq0$.}}
\end{cor}
\begin{proof} By \cite[Theorem 3.5]{BNS1} and \cite[Theorem 2.3]{NS}, one obtains that an $R$-module $L$ is $\mathfrak{a}$-cofinite if and only if $\mathrm{Ext}^i_R(R/\mathfrak{a},L)$ is finitely generated for $i=0,1,2$, which implies that $\mathrm{H}^{t}_\mathfrak{a}(N)$ are $\mathfrak{a}$-cofinite if and only if $\mathrm{Hom}_R(R/\mathfrak{a},\mathrm{H}^{t}_\mathfrak{a}(N))$ are finitely generated for all $t\geq0$. Now the statement follows from
Theorem \ref{lem:1.0}.
\end{proof}

The next result has been proved in \cite[Theorem 2.21]{VA}, we gives a new proof method.

\begin{prop}\label{lem:2.4}{\it{Let $M$ and $N$ be two
$R$-modules with $M$ finitely generated and $s,t$ be non-negative integers such that

$(1)$ $\mathrm{Ext}^{s+t-i}_R(M,\mathrm{H}^{i}_\mathfrak{a}(N))=0$ for all $0\leq i< t$ or $t+1\leq i\leq s+t$;

$(2)$ $\mathrm{Ext}^{s+1+i}_R(M,\mathrm{H}^{t-i}_\mathfrak{a}(N))=0$ for all $1\leq i\leq t$;

$(3)$ $\mathrm{Ext}^{s-1-i}_R(M,\mathrm{H}^{t+i}_\mathfrak{a}(N))=0$ for all $1\leq i\leq s-1$.\\
Then $\mathrm{Ext}^{s}_R(M,\mathrm{H}^{t}_\mathfrak{a}(N))\cong\mathrm{H}^{s+t}_\mathfrak{a}(M,N)$.}}
\end{prop}
\begin{proof} Consider the Grothendieck spectral sequence \begin{center}
$\xymatrix@C=15pt@R=5pt{
 E_2^{p,q}=\mathrm{Ext}^{p}_R(M,\mathrm{H}^{q}_\mathfrak{a}(N))\ar@{=>}[r]_{\ \ \ \ \ \ \  p}&
 \mathrm{H}^{p+q}_\mathfrak{a}(M,N).}$\end{center} There is a finite filtration\begin{center}$
 0=\Phi^{s+t+1}H^{s+t}\subseteq \Phi^{s+t}H^{s+t}\subseteq\cdots\subseteq \Phi^{1}H^{s+t}\subseteq \Phi^{0}H^{s+t}=H^{s+t}:=\mathrm{H}^{s+t}_\mathfrak{a}(M,N)$,
\end{center}so that $\Phi^{p}H^{s+t}/\Phi^{p+1}H^{s+t}\cong E_\infty^{p,s+t-p}$ for $0\leq p\leq s+t$.
Let $r\geq 2$. Consider the differential \begin{center}$E_r^{s-r,t+r-1}\xrightarrow{d_r^{s-r,t+r-1}}E_r^{s,t}
\xrightarrow{d_r^{s,t}}E_r^{s+r,t-r+1}.$
\end{center}
By conditions (2) and (3), we have $E_r^{s-r,t+r-1}=0=E_r^{s+r,t-r+1}$ for $r\geq 2$. By condition (1)
we have $0=\Phi^{s+t+1}H^{s+t}=\cdots=\Phi^{s+1}H^{s+t}$ and $\Phi^{s}H^{s+t}=\cdots=\Phi^{0}H^{s+t}=\mathrm{H}^{s+t}$.
So
\begin{center}$\mathrm{Ext}^{s}_R(M,\mathrm{H}^{t}_\mathfrak{a}(N))=E_{2}^{s,t}\cong E_{\infty}^{s,t}\cong \Phi^{s}H^{s+t}=\mathrm{H}^{s+t}=\mathrm{H}^{s+t}_\mathfrak{a}(M,N)$.\end{center}
We get the isomorphism we seek.
\end{proof}

\begin{cor}\label{lem:2.5} {\rm (\cite{GC}).} {\it{Let $M$ and $N$ be two
$R$-modules with $M$ finitely generated. Suppose that $\mathrm{pd}_RM=s<\infty$ and $\mathrm{cd}(\mathfrak{a},N)=t<\infty$.
Then $\mathrm{Ext}^{s}_R(M,\mathrm{H}^{t}_\mathfrak{a}(N))\cong\mathrm{H}^{s+t}_\mathfrak{a}(M,N)$. In particular, $\mathrm{H}^{s+t}_\mathfrak{a}(M,N)$ is artinian $\mathfrak{a}$-cofinite.}}
\end{cor}

\begin{cor}\label{lem:2.6} {\it{Let $M$ and $N$ be two
$R$-modules with $M$ finitely generated and $t$ be non-negative integers such that $\mathrm{H}^{i}_\mathfrak{a}(N)=0$ for all $i\neq t$.
Then $\mathrm{Ext}^{s}_R(M,\mathrm{H}^{t}_\mathfrak{a}(N))\cong\mathrm{H}^{s+t}_\mathfrak{a}(M,N)$ for all $s\geq0$.}}
\end{cor}

\end{document}